\documentclass[epic,eepic,11pt]{amsart}

\usepackage{amsfonts,mathrsfs}
\usepackage[mathscr]{eucal}
\usepackage{mathtools}
\usepackage{amssymb}
\usepackage{amscd}
\usepackage{fancyhdr}
\usepackage[all,cmtip]{xy}

\usepackage{euler,eucal}

\DeclareMathAlphabet{\mathbf}{T1}{ppl}{bx}{n}
\DeclareMathAlphabet{\mathrm}{T1}{ppl}{m}{n}

\numberwithin{equation}{section}

\newcommand\note[1]%
{$^\dagger$\marginpar{\footnotesize{$^\dagger${#1}}}}

\def\({\left(}
\def\){\right)}
\def\<{\left<}
\def\>{\right>}

\newtheorem{theorem}{Theorem}[section]
\newtheorem{proposition}[theorem]{Proposition}
\newtheorem{lemma}[theorem]{Lemma}
\newtheorem{definition}[theorem]{Definition}

\newtheorem{corollary}[theorem]{Corollary}

\theoremstyle{definition}
\newtheorem{example}[theorem]{Example}
\newtheorem{remark}[theorem]{Remark}

\newcommand\bb[1]{{\text{\bf#1}}}

\newcommand\Z{\bb{Z}}

\newcommand\R{\mathbb{R}}

\newcommand     {\comment}[1]   {}
\newcommand{\mute}[2] {}
\newcommand     {\printname}[1] {}

\newcommand\funclim[1]{\operatorname*{\mathrm{#1}}}

\renewcommand\lim{\funclim{lim}}

\newcommand\sur{\mathrel{\to\kern-1.8ex\to}}
\newcommand\iso{\mathrel{\hookrightarrow\kern-1.8ex\to}}

\newcommand\longhookrightarrow{\lhook\joinrel\longrightarrow}

\newcommand\longsur{\mathrel{\longrightarrow\kern-1.8ex\to}}
\newcommand\longiso{\mathrel{\longhookrightarrow\kern-1.8ex\to}}

\begin{document}

\bibliographystyle{amsalpha}
\date{\today}

\title{Hodge theory on transversely symplectic foliations}

\author{Yi Lin }

 %\dedicatory{}
 \date{\today}

\maketitle
\begin{abstract}
In this paper, we develop symplectic Hodge theory on transversely symplectic foliations. In particular, we establish the symplectic $d\delta$-lemma for any such foliations with the (transverse) $s$-Lefschetz  property.  As transversely symplectic foliations include many geometric structures, such as  contact manifolds, co-symplectic manifolds, symplectic orbifolds, and symplectic quasi-folds as special examples, our work provides a unifying treatment of symplectic Hodge theory in these geometries. 

As an application, we show that on compact $K$-contact manifolds, the $s$-Lefschetz property implies a general result on the vanishing of cup products, and that the cup length of a $2n+1$ dimensional compact $K$-contact manifold with the (transverse) $s$-Lefschetz property is at most $2n-s$. For any even integer $s\geq 2$, we also apply our main result to produce examples of $K$-contact manifolds that are $s$-Lefschetz but not $(s+1)$-Lefschetz.  

%In this paper, on any compact symplectic manifold we construct a sub-complex of the distributional De Rham complex.
%This subcomplex carries a $sl_2$ sub-module structure of the distributional De Rham complex, and computes the De Rham cohomology.
%For any compact symplectic manifold with the Hard Lefschetz property, we prove that there is a Poincar\'{e} duality between the

\end{abstract}

%\tableofcontents

%\addcontentsline{toc}{section name}{entry text}
%named Preface in table of contents.
% \subsection*{}

\setcounter{section}{0} \setcounter{subsection}{0}

\section{Introduction}

Hodge theory on symplectic manifolds was introduced by Ehresmann and Libermann \cite{EL49}, \cite{L55}, and was rediscovered by Brylinski \cite{brylinski;differential-poisson}.  Brylinski proved that on a compact K\"ahler manifold, any de Rham cohomology class admits a symplectic harmonic representative, and further conjectured that on a compact symplectic manifold, any de Rham cohomology class has a symplectic harmonic representative.

However, Mathieu proved that for a $2n$ dimensional symplectic manifold $(M,\omega)$, the Brylinski conjecture is true if and only if $(M,\omega)$ satisfies the Hard Lefschetz property, i.e., for any $0\leq k\leq n$, the Lefschetz map

\begin{equation}\label{Lefschetz-map}L^k: H^{n-k}(M)\rightarrow H^{n+k}(M),\,\,\,\,[\alpha]\mapsto [\omega^k\wedge \alpha]\end{equation}
is surjective. Mathieu's result was sharpened by Merkulov \cite{Mer98} and Guillemin\cite{Gui01}, who independently established the symplectic $d\delta$-lemma. 

 On symplectic manifolds, Fern\'{a}ndez, Mu\~{n}oz and Ugarte \cite{FMU04} introduced a notion of weakly Lefschetz property. More precisely, for any $0\leq s\leq n-1$, a $2n$ dimensional symplectic manifold $(M,\omega)$ is said to satisfy the $s$-Lefschetz property, if and only if for any $0\leq k\leq s$, the Lefshetz map (\ref{Lefschetz-map}) is surjective.  
 Fern\'{a}ndez, Mu\~{n}oz and Ugarte extended the symplectic $d\delta$-lemma to symplectic manifolds with the $s$-Lefschetz property;  moreover, 
 for any even integer $s\geq 2$, they also produced examples of symplectic manifolds which are $s$-Lefschetz but not $(s+1)$-Lefschetz, c.f. \cite{FMU04}, \cite{FMU07}.
 
 The foliated version of symplectic Hodge theory has also been studied in the literature. Pak \cite{Pak08} extended Mathieu's Lefschetz theorem to the basic cohomology groups of transversely symplectic flows.  In the context of odd dimensional symplectic manifolds, He \cite{He10}  further developed Hodge theory on transversely symplectic flows, and established the symplectic $d\delta$-lemma in this framework. 
 
Replacing the Riemannian hodge theory used in \cite{CNY13} by symplectic Hodge theory developed by He \cite{He10}, Lin \cite{L13} proved that for a compact $K$-contact manifold, the transverse Hard Lefschetz property on basic cohomology groups is equivalent to the Hard Lefschetz property on de Rham cohomology groups introduced in \cite{CNY13}. In particular, this implies immediately  that on a compact Sasakian manifold, the two existing versions of Hard Lefschetz theorems,  established in \cite{ka90} and \cite{CNY13} respectively,  are mathematically equivalent to each other.  Boyer and Galicki \cite{BG08} raised the open question of whether there exist simply-connected $K$-contact manifolds which do not admit any Sasakian structures. Lin's Hodge theoretic methods shed new insights on the Lefschetz property of a $K$-contact manifold, and allow him to produce such examples in any dimension $\geq 9$.

However,  many singular symplectic spaces naturally arise as the leaf space of higher dimensional transversely symplectic foliations.  For example,  effective symplectic orbifolds in the sense of Satake \cite{S57}, c.f. Example \ref{orbifolds}, and symplectic quasi-folds introduced by Prato\cite{P01}, c.f. Example \ref{quasi-folds}.  Thus it would be natural to consider  Hodge theory on transversely symplectic foliations of arbitrary dimension as well. In the present paper, for any transversely symplectic foliation,  we introduce the notion of transverse $s$-Lefschetz property on its basic cohomology groups ( see Definition \ref{weak-lefschetz-property}), and generalize the machinery of symplectic Hodge theory to this framework. Among other things, we prove the symplectic $d\delta$-lemma in this setup.

We then apply our results to the study of $K$-contact manifolds, and prove that the transverse $s$-Lefschetz property imposed on the basic cohomology of a $K$-contact manifold is equivalent to the $s$-Lefschetz property imposed on its de Rham cohomology (see Theorem \ref{main-result1}).   As a first application, we show that on compact $K$-contact manifolds, the $s$-Lefschetz property implies a fairly general result on the vanishing of cup products, which implies that the cup length of a $2n+1$ dimensional compact $K$-contact manifold with the $s$-Lefschetz property is at most $2n-s$. It is well known that the cup length of a $2n+1$ dimensional compact $K$-contact manifold is at most $2n$, c.f. \cite[Theorem 7.4.1]{BG08}.  More recently, using the methods of rational homotopy theory, it is shown in \cite{MT15} that a simply-connected $7$ dimensional compact Sasakian manifold has vanishing cup product $H^2\times H^2\rightarrow H^4$.  Our result does not assume the $K$-contact manifold under consideration to be simply-connected, and simultaneously generalizes all these known results in the literature.  

As a second application, for any even integer $s\geq 2$, we produce examples of $2s+5$ dimensional compact $K$-contact manifolds that are $s$-Lefschetz but not $(s+1)$-Lefschetz. In particular, our  construction provides new examples of simply-connected compact $K$-contact manifolds that do not admit any Sasakian structures in dimension $\geq 9$.

Our paper is organized as follows.  Section \ref{transverse-sym} reviews  preliminaries on transversely symplectic foliations.
Section \ref{transverse-sym-Hodge} explains how to do Hodge theory on transversely symplectic foliations. Section \ref{ddelta-lemma} proves the symplectic $d\delta$-lemma for a transversely symplectic foliation. Section \ref{review-contact}  recalls necessary background materials on $K$-contact and Sasakian geometries. Section \ref{Kcontact-s-lefschetz} proves that the transverse $s$-Lefschetz condition imposed on the basic cohomology groups of a compact $K$-contact manifold is equivalent to the $s$-Lefschetz condition imposed on its de Rham cohomology groups. Section \ref{cup-length} shows that the cup length of a compact $2n+1$ dimensional $K$-contact manifold with the $s$-Lefschetz property is at most $2n-s$. Section \ref{main-examples} constructs, for any even integer $s\geq 2$, a compact $K$-contact manifold that is $s$-Lefschetz but not $(s+1)$-Lefschetz.
 
 \subsection*{Acknowledgement} This work was completed when the author was visiting the Department of Mathematics at Cornell University in the fall of 2015, and the School of Mathematics at Sichuan University in the spring of 2016. He would like to thank these two institutions for providing him an excellent working environment during his visit.

\section{ Review of transversely symplectic foliations }\label{transverse-sym}

%\begin{definition} A co-dimension $q$ \textbf{foliation} $\mathcal{F}$ on an $n$ dimension manifold $M$ is given by a Haefliger cocycle $\{U_i, \varphi_i, f_{ij}\}$ such that 
%\begin{itemize}

%\item [1)]  $\{U_i\}_{i\in I}$ is an open cover of $M$;
%\item [2)] $\forall\, i$, $\varphi_i:U_i\rightarrow \mathbf{R}^q$ is a submersion from $U_i$ onto an open subset $\varphi_i(U_i)$ in $\mathbf{R}^q$;
%\item [3)] The maps $f_{ij}:\varphi_{ij}(U_i\cap U_j)\rightarrow \varphi_j(U_i\cap U_j)$
%are diffemorphisms that satisfy
%\begin{equation}\label{cocycle-condition} \varphi_j =f_{ij}\circ \varphi_i,\end{equation}
%whenever $U_i\cap U_j\neq \emptyset$.
%\end{itemize}

Let $\mathcal{F}$ be a foliation of co-dimension $q$ on a smooth manifold $M$ of dimension $n$, and let $P$ be the integrable subbundle of $TM$ associated to $\mathcal{F}$.  A \emph{foliation chart} of $(M,\mathcal{F})$ is a 
coordinate chart $(\varphi: U\rightarrow \mathbf{R}^{n-q}\times \mathbf{R}^{q})$ on $M$, such that for any $z\in U$, a vector $X_z\in T_zM$ is tangent to  the fiber $P_z$ if and only if $\varphi_{*z}(X_z)$ is tangent to the vertical plaque $\mathbf{R}^{n-q}\times \{ \pi\circ \varphi(z)\}$, where $\pi: \mathbf{R}^{n-q}\times\mathbf{R}^q\rightarrow \mathbf{R}^q$ is the projection map.  Let $\varphi: U\rightarrow \mathbf{R}^{n-q}\times \mathbf{R}^{q}$ be a foliation chart, and let $y_i$ be the $i$-th coordinate of the function $\pi\circ \varphi$. Then we call $\{y_1,\cdots, y_q\}$ \emph{transverse  coordinates} on $U$.
%For  any $x\in M$, denote by $\mathcal{F}(x)$ the leaf passing through $x$. Then $P=\cup_{x\in M}T_x\mathcal{F}(x)$ is an integrable subbundle of $TM$.

On the foliated manifold $(M,\mathcal{F})$,  the spaces of \emph{horizontal forms} $\Omega_{hor}(M)$ and \emph{basic forms}  $\Omega_{bas}(M)$ are defined as follows
respectively.
\begin{equation}\begin{split} &\Omega_{\textmd{hor}}(M)=\{\alpha\in\Omega(M)\mid\iota_{X}\alpha=0,\,\forall\, X\in \Gamma(P)\}, \\
&
\Omega_{\textmd{bas}}(M)=\{\alpha\in\Omega(M)\mid\iota_{X}\alpha=0,\quad\mathcal{L}_{X}\alpha=0,\forall X\in \Gamma(P)\}.
\end{split}\end{equation}
Here $\Gamma(P)$ denotes the space of smooth sections of the integrable subbundle $P$ associated to the foliation. Since the exterior differential operator $d$ preserves basic forms,  we obtain a subcomplex of the de Rham complex $\{\Omega^{*}(M),d\}$,
which is called the \emph{basic de Rham complex}
$$
\xymatrix@C=0.5cm{
  \cdot\cdot\cdot \ar[r] & \Omega^{k-1}_{\textmd{bas}}(M) \ar[r]^{d} & \Omega^{k}_{\textmd{bas}}(M) \ar[r]^{d} & \Omega^{k+1}_{\textmd{bas}}(M) \ar[r]^{\,\,\,\,\,d} & \cdot\cdot\cdot }
$$
The cohomology of the basic de Rham complex $\{\Omega^{*}_{\textmd{bas}}(M),d\}$, denoted by $H^{*}_{B}(M,\R)$, is called the \emph{basic de Rham cohomology} of $M$.
If $M$ is connected\footnote{The manifolds we consider in this paper are all connected.}, then $H^{0}_{B}(M,\R)\cong\mathbb{R}^{1}$. Moreover, the inclusion $\Omega_{bas}^1(M)\hookrightarrow \Omega^1(M)$ induces an injective map 
$H^1_B(M)\rightarrow H^1(M)$ (\cite[Prop. 4.1]{T97}). In general, the group $H^{k}_{B}(M,\R)$ may be infinite-dimensional for $k\geq2$.

\begin{definition}[Transversely symplectic foliation]
Let $M$ be a manifold with a foliation $\mathcal{F}$. A closed two form $\omega$ is said to be transversely symplectic with respect to $\mathcal{F}$, if for any $p\in M$, the kernel of $\omega_p$ coincides with the tangent space of the leave passing through $p$.  A \emph{transversely symplectic foliation} is a triple $(M,\mathcal{F},\omega)$, such that $\omega$ is transversely symplectic with respect to the foliation $\mathcal{F}$.
\end{definition}

\begin{example}\label{Contact-example}(\textbf{Contact manifolds})

On any co-oriented contact manifold $(M,\eta)$ with a contact one form $\eta$, there is a nowhere vanishing Reeb vector $\xi$ uniquely determined by the  equations $\iota_{\xi}\eta=1,\,\,\,\,\,\iota_{\xi}d\eta=0$.

The one dimensional foliation $\mathcal{F}_{\xi}$ associated to $\xi$ is called the Reeb characteristic foliation. It is easy to see that $\mathcal{F}_{\xi}$  is transversely symplectic with $\omega:=d\eta$ being the transversely symplectic form.

\end{example}

\begin{example}\label{Co-sympl-example}(\textbf{Co-symplectic manifolds})

A  $2n+1$ dimensional manifold $M$ is co-symplectic, if it is equipped with a closed two form $\omega$, and a closed one form $\eta$, such that $\omega^n\wedge \eta\neq 0$. The Reeb vector field $\xi$ is uniquely determined by $\iota_{\xi}\eta=1$, $\iota_{\xi}d\eta=0$.  Let $\mathcal{F}_{\xi}$ be the associated Reeb characteristic foliation. Then
$(M,\mathcal{F},\omega)$ is transversely symplectic.

\end{example}

\begin{example} (\textbf{Odd dimensional symplectic manifolds}, \cite{He10})

Let $M$ be a $2n+1$ dimensional manifold with a volume form $\Omega$ and a closed 2-form $\omega$,
such that $\omega^{n}\neq0$ everywhere.
Then the triple $(M,\omega,\Omega)$ is called an $(2n+1)$ dimensional symplectic manifold.
In particular, there exists a unique canonical vector field $\xi$ defined by
$$
\iota_{\xi}\omega=0,\quad\quad\quad\iota_{\xi}\Omega=\frac{\omega^{n}}{n!}
$$
which is called the Reeb vector field on $(M,\omega,\Omega)$.
The Reeb vector field yields a canonical 1-dimensional foliation $\mathcal{F}_{\xi}$
. It is straightforward to check that $(M,\omega, \mathcal{F}_{\xi})$ is
transversely symplectic.
\end{example}
\begin{example}(\textbf{Symplectic orbifolds})\label{orbifolds}
Let $(X,\sigma)$ be a $2n$ dimensional effective symplectic orbifold in the sense of Satake \cite{S57}.
Then the total space of the orthogonal frame orbi-bundle $\pi:P\longrightarrow X$ is a smooth manifold, on which the structure group $O(2n)$ acts locally free. Let $\mathcal{F}$ be the foliation induced by the $O(2n)$ action, whose leaves are precisely given by the orbits of the $O(2n)$ action. Set $\omega:=\pi^{*}\sigma$. Then it is easy to see that $(P,\mathcal{F},\omega)$ 
is a transversely symplectic foliation. In this case, the leaf space of $\mathcal{F}$ can be naturally identified with the orbifold $X$.
\end{example}
\begin{example}(\textbf{Symplectic quasi-folds} \cite{P01})\label{quasi-folds}
Suppose that a torus $T$ acts on a symplectic manifold $(X,\sigma)$ in a Hamiltonian fashion with a moment map
$
\phi:X\longrightarrow\mathfrak{t}^{*}. $ Let $N\subset T$ be a non-closed subgroup with Lie algebra $\mathfrak{n}$. Then the action of $N$ on $X$ has a moment map   
$
\varphi:X\longrightarrow\mathfrak{n}^{*}
$, which is given by the composition of $\phi: X\rightarrow \mathfrak{t}^*$ and the natural projection map $\mathfrak{t}^*\rightarrow \mathfrak{n}^*$. 

Let $a$ be a regular value of $\varphi$. Consider the the submanifold $M=\varphi^{-1}(a)\subset X$.
The $N$-action on $M$ yields a transversely symplectic foliation $\mathcal{F}$
with the transversely symplectic form $\omega:=i^{*}\sigma$,
where $i: M\hookrightarrow X$ is the inclusion map.
When $N$ is a connected subgroup, the leaf space of $\mathcal{F}$ is exactly the so called symplectic quasi-fold introduced by E. Prato \cite{P01}.
\end{example}

\begin{theorem} \label{Darboux-thm}(Darboux theorem) Let $(\mathcal{F},\omega)$ be a transversely symplectic foliation of co-dimension  $2n$ on a $2n+l$ dimensional manifold $M$. Then for any $p\in M$, there exists a foliation chart $(U, \psi)$ around $p$ equipped with transverse coordinates $\{ x_1, y_1,\cdots x_n,y_n\}$, such that
\[ (\psi^{-1})^*\omega\vert_{\psi(U)}= \displaystyle \sum_{i=1}^n dx_i\wedge dy_i.\]
For a transversely symplectic foliation, such transverse coordinates will be called \textbf{transverse Darboux coordinates}.
\end{theorem}
\begin{proof}  $\forall\, p\in M$,  choose a foliation coordinate chart $(U,\varphi)$ around $p$. Without loss of generality, we may assume that $\varphi(U)=V_1\times V_2 \subset \mathbf{R}^l\times \mathbf{R}^{2n}$ for two open subsets $V_1$ and $V_2$ in $\mathbf{R}^l$ and $\mathbf{R}^{2n}$ respectively.  Let $\pi: \varphi(U)\rightarrow V_2$ be the projection map, and let $\{z_1, \cdots, z_l, w_1, \cdots, w_{2n}\}$ be the foliation coordinates on $\varphi(U)$.

Note that by definition, $\omega$ is a basic form. It follows easily that $(\varphi^{-1})^*\omega$ induces a symplectic form $\sigma$ on $V_2$.  Replacing $V_2$ by a smaller open subset if necessary, we may assume that there are Dauboux coordinates $\{x_1,y_1,\cdots, x_n,y_n\}$ on $V_2$, such that 
\[ \sigma =\displaystyle \sum_{i=1}^n dx_i\wedge dy_i, \,\,\,\text{ on }\,\, V_2.\]

Now let $F: V_1\times V_2\rightarrow V_1\times V_2$ be the diffeomorphism given by the change of coordinates map
\[ (z_1, \cdots, z_l, w_1, w_2,\cdots,w_{2n-1}, w_{2n})\mapsto (z_1, \cdots, z_l, x_1, y_1, \cdots, x_n, y_n),\] 
and let $\psi= F\circ \varphi$. Then it is easy to see that the foliation chart $(U,\psi)$ has the desired property stated in
Theorem \ref{Darboux-thm}.

\end{proof}

\section{Transverse symplectic Hodge theory}\label{transverse-sym-Hodge}

In this section we develop the machinery of symplectic Hodge theory on a transversely symplectic foliation. We refer to \cite{brylinski;differential-poisson} and \cite{Yan96}
for general background on symplectic Hodge theory. We need to explain how to define the symplectic Hodge star operator on the space of basic forms. To this end, we first review the construction of symplectic Hodge star operator on a symplectic vector space.

Let $(V,\sigma)$ be a symplectic vector space, where $\sigma$ is a non-degenerate bi-linear pairing on $V$. Since $\sigma$ is non-degenerate, it induces a linear isomorphism 
\[  V\rightarrow  V^*, \,\,\, X\mapsto \iota_X\sigma.\]

Its inverse  map extends linearly to a linear isomorphism $ \sharp: \wedge^k V^*\rightarrow \wedge^k V$. Thus we get a bi-linear pairing 
\[B(\cdot, \cdot): \wedge^k V^*\times \wedge^kV^* \rightarrow \mathbf{R},\,\,\,\,B(\alpha, \beta)=<\sharp (\alpha),\beta>,\]
where $\alpha$ and $\beta$ are $k$-forms on $V$, and $<\cdot, \cdot>$ is the natural dual pairing between $k$-forms and $k$-vectors.  It is straightforward to check that this pairing is non-degenerate. In this context, the symplectic Hodge star of a $k$-form $\alpha$ on $V$ is uniquely determined by the following equation.

\[ \beta\wedge \star \alpha  =B(\beta, \alpha)\dfrac{\sigma^n}{n!}, \,\,\,\, \forall\,\beta \in \wedge^k V^*.\]

We recall that the following identify holds for symplectic Hodge star operator.

\begin{equation}\label{star-square}  \star ^2 =\text{id}.
\end{equation} 
Now let $(\mathcal{F},\omega)$ be a transversely symplectic foliation of co-dimension $2n$ on a manifold $M$, let $P$ be the integrable sub-bundle of $TM$ associated to the foliation $\mathcal{F}$, and Let $Q=TM/ P$.  Then the projection map $\pi: TM\rightarrow Q$ induces a pullback map $\pi^*: \wedge^r Q\rightarrow \wedge^r T^*M$. Clearly, if $s$ is a section of $\wedge^rQ$, then $\pi^* (s)$ is a horizontal $r$-form on $M$. Conversely, if $\alpha$ is a horizontal $r$-form on $M$, then there exists a unique section $s$ of $\wedge^rQ$, such that $\alpha=\pi^* s$.

In particular, since $\omega$ is horizontal, there is a section $\sigma \in \wedge^2Q$ such that $\pi^*\sigma=\omega$. It is easy to see that at any point $x\in M$, $(Q_x,\sigma_x)$ is a symplectic vector space. Thus there is a (point-wise defined) symplectic Hodge star operator on the space of sections of $\wedge^*Q$. By identifying horizontal forms on $M$ with sections of $\wedge^*Q$, we get a symplectic Hodge star operator 
\[\star: \Omega^k_{hor}(M)\rightarrow \Omega^{2n-k}_{hor}(M).\]

\begin{lemma} \label{basic-star} If $\alpha$ is a basic $k$-form, then $\star \alpha$ is a basic $(2n-k)$-form. 
\end{lemma}

\begin{proof}  To show that $\star \alpha$ is basic, it suffices to show that for any $p\in M$, there exists a foliation coordinate neighborhood $U$ of $p$, such that $\star \alpha$ has the following local expression on $U$.
\[ \star \alpha = \displaystyle \sum_I g(x_1, y_1, \cdots, x_n, y_n) dx_I\wedge dy_J,\]
where $\{x_1, y_1,\cdots, x_n,y_n\}$ are transverse Darboux coordinates on $U$, and $I$ and $J$ are multi-index.

Since $\alpha$ is a basic form, for any $p\in M$, there exists a foliation coordinate  neighborhood $U$ of $p$, equipped with transverse Darboux coordinates $\{x_1, y_1, \cdots, x_n, y_n\}$, such that 
\[ \alpha \vert_U= \displaystyle \sum f(x_1,y_1,\cdots, x_n, y_n)dx_I\wedge dy_J,\]
where $I$ and $J$ are multi-index. Without loss of generality, we may assume that
$\alpha\vert_U=f(x_1,y_1,\cdots, x_n, y_n)dx_I\wedge dy_J$ for some given multi-index
$I$ and $J$.

Since \[\omega\vert_U=\displaystyle \sum_{i=1}^n dx_i\wedge dy_i,\]
a straightforward calculation shows that
\[ \star \alpha=c f(x_1, y_1, \cdots, x_n, y_n) dx_{I'}\wedge dy_{J'},\]
where $c$ is a constant, and $I'$ and $J'$ are multi-indexes.  
Therefore $\star \alpha$ must also be a basic form.

\end{proof}

% \begin{remark}
% We say that a one form $\lambda\in \Omega(M)$ is a connection $1$-form if $\iota_{\xi}\lambda=1$ and if $\mathcal{L}_{\xi}\lambda=0$.
% It is shown in \cite{He10} that if $M$ is compact, and if there is a connection $1$-form on $M$, then $\omega^k$ always represents a non-trivial cohomology class
% in $H^{2k}_B(M,\R)$. Furthermore, we always have $\Omega=\lambda\wedge\frac{\omega^n}{n!}$. Clearly, if $M$ is a contact manifold with a contact one form $\eta$, then $\eta$ will be a connection $1$-form on $M$. If $M$ is also compact, then $\omega^k$ always represents a non-trivial cohomology class in $H^{2k}_B(M,\R)$.
%\end{remark}

There are three important operators, the Lefschetz map $L$, the dual Lefschetz map $\Lambda$, and the degree counting map $H$, which are defined on horizontal forms as follows.

\begin{equation} \label{three-canonical-maps} \begin{aligned} & L :\Omega_{hor}^*(M) \rightarrow \Omega_{hor}^{*+2}(M), \,\,\,\alpha \mapsto \alpha \wedge \omega,\\
 & \Lambda: \Omega_{hor}^*(M) \rightarrow \Omega_{hor}^{*-2}(M),\,\,\,\alpha \mapsto \star L\star \alpha,\\
& H: \Omega^k_{hor}(M)\rightarrow \Omega^k_{hor}(M),\,\,\,H(\alpha)=(n-k)\alpha,\,\,\,\alpha \in \Omega^{k}_{hor}(M).\end{aligned}\end{equation}

Using transverse Ddarboux coordinates, we have the following local description of the operator $\Lambda$. We refer to \cite[Lemma 2.1.4]{He10} for a proof.

\begin{lemma} \label{canonical-local-form-Lambda} Let $(U, x_1,\cdots, x_n,y_1,\cdots, y_n)$ be a transverse Darboux coordinate chart on $M$. Then for any form $\alpha\in \Omega^*(U)$, we have
\[ \Lambda \alpha=\displaystyle \sum_{i=1}^n\iota_{\frac{\partial}{\partial x_i}} \iota_{\frac{\partial}{\partial y_i}}\alpha.\] 
\end{lemma}

Exploring the one to one correspondence between horizontal forms on $M$ and sections of $\wedge^* Q$, it is easy to see that the actions of $L$, $\Lambda$ and $H$ on horizontal satisfy the following commutator relations.

\begin{equation} \label{sl2-module-on-forms} [  \Lambda, L]=H, \,\,\,[H, \Lambda]=2\Lambda,\,\,\,[H, L]=-2L.
\end{equation}

Note that $\omega$ is a basic form itself.  As an immediate consequence of Lemma \ref{basic-star}, we have the following result. 

\begin{lemma} \label{basic-sl2} The operators $L$, $\Lambda$ and $H$ map basic forms to basic forms, and satisfy all three commutator relations in Equation \ref{sl2-module-on-forms}.
\end{lemma} 

Therefore, these three operators define a representation of the Lie algebra $sl(2)$ on $\Omega_{bas}(M)$. Although the $sl_2$-module
$\Omega_{bas}(M)$ is infinite dimensional, there are only finitely many eigenvalues of the operator $H$. Such an $sl_2$-module is said to be of finite $H$-type, and has been studied in great details in \cite{Ma95} and \cite{Yan96}.  Applying  \cite[Corollary 2.6]{Yan96} to the present situation, we have the following result.

\begin{lemma} \label{yan's-result}  Let $(M,\mathcal{F},\omega)$ be a transversely symplectic foliation of co-dimension $2n$. For any $0\leq k\leq n$, $\alpha\in \Omega_{bas}^k(M)$ is said to be primitive if  $L^{n-k+1}\alpha=0$. Then we have that
\begin{enumerate}

 \item [a)]
 a basic $k$-form $\alpha$ is primitive if and only if $\Lambda \alpha=0$;
\item [b)] For any $0\leq k\leq n$, the Lefschetz map
\[ L^{n-k}: \Omega_{bas}^k(M)\rightarrow \Omega_{bas}^{2n-k}(M),\,\,\,\,\alpha\mapsto \omega^{n-k}\wedge\alpha\]
is a linear isomorphism.

 \item[c)] any differential form
$\alpha_k \in \Omega_{bas}^k(M)$ admits a unique Lefschetz decomposition
\begin{equation}\label{lefschetz-decompose-forms} \alpha_k =\displaystyle \sum_{r\geq \text{max}(\frac{k-n}{2}, 0)} \dfrac{L^r}{r!}\beta_{k-2r},\end{equation}
where $\beta_{k-2r}$ is a primitive basic form of degree $k-2r$.

\end{enumerate}
\end{lemma}
\begin{remark}Throughout the rest of this paper, we will denote the space of primitive basic
 $k$-forms on $M$ by $\mathcal{P}_{bas}^k(M)$.

\end{remark}

%Since the kernel of the $2$-form $\omega$ is spanned by the Reeb vector $\xi$, it induces a non-degenerate pairing $G(\cdot,\cdot)$ on $\Omega_{hor}^k(M)$. On the space of horizontal $k$-forms, the symplectic Hodge star  is defined as follows.
%\[  \star\alpha_k \wedge \beta_k= G(\alpha_k,\beta_k)\dfrac{\omega^n}{n!},\]
%where $\alpha_k, \beta_k\in \Omega^{k}_{hor}(M)$.

%It is easy to check that the symplectic Hodge star operator maps basic forms to basic forms. So there is a symplectic Hodge star operator on the space of basic forms.
%\[ \star: \Omega_{bas}^{k}(M)\rightarrow \Omega^{2n-k}_{bas}(M).\]

Applying \cite[Theorem 3.16]{W80} to the present situation, we have the following Weil identity.

\begin{lemma}\label{Weil-identity}Let $(M,\omega,\Omega)$ be a transversely symplectic foliation of co-dimension $2n$, and let $\alpha$ be a primitive
$k$-form in $\Omega_{bas}^k(M)$.Then for any $r\leq n-k$,
\[\star L^r\alpha=(-1)^{\frac{k(k-1)}{2}}\dfrac{r!}{(n-k-r)!}L^{n-k-r}\alpha.\]
\end{lemma}

The symplectic Hodge operator gives rise to the symplectic adjoint operator $\delta$ of the exterior differential $d$ as follows.
 \[ \delta \alpha_k=(-1)^{k+1}\star d\star\alpha_k,\,\,\,\alpha_k\in\Omega_{bas}^k(M).\]

It is easy to see that $\delta \circ \delta=0$. Thus we have the following differential complex of homology.

$$
\xymatrix@C=0.5cm{
  \cdot\cdot\cdot \ar[r] & \Omega^{k+1}_{\textmd{bas}}(M) \ar[r]^{\delta} & \Omega^{k}_{\textmd{bas}}(M) \ar[r]^{\delta} & \Omega^{k-1}_{\textmd{bas}}(M) \ar[r]^{\,\,\,\,\,\delta} & \cdot\cdot\cdot }
$$

The homology of the above complex, denoted by $H_{\delta}(\Omega_{bas}(M),\R)$, is called the \emph{basic $\delta$-homology} of the foliated manifold
$M$. We note that the symplectic hodge star operator maps a $\delta$-closed basic form of degree $k$ to a $d$-closed differential basic form of degree $2n-k$. Using the identity $\star^2=\text{id}$, it is straightforward to show the following result.

\begin{lemma}\label{homology-cohomology} The symplectic Hodge star operator induces a linear isomorphism
\[ \star: H_{\delta}(\Omega_{bas}^k(M),\R) \cong H^{2n-k}_B(M,\R) .\]
\end{lemma}

It is noteworthy that the symplectic adjoint operator $\delta$ anti-commutes with $d$. In this context, a basic form $\alpha$ is said to be symplectic harmonic if and only if $d\alpha=\delta\alpha=0$. We define 

\[\begin{split}&\Omega_{hr}^k(M)=\{\alpha\in \Omega^k_{bas}(M)\,\vert\, d\alpha=\delta\alpha=0\},
\\& H_{hr}^k(M,\R)=\dfrac{ \text{Im}\left(\Omega^k_{hr}(M)\hookrightarrow\Omega^k_{bas}(M)\right)}{\text{im} (\Omega_{bas}^{k-1}(M)\xrightarrow{d}\Omega^k_{bas}(M))}.\end{split}\]

It is easy to see that the operators $L$, $\Lambda$ and $H$ map harmonic forms to harmonic forms, and therefore turn $\Omega_{hr}^*(M)$ into a $sl_2$ module of finite $H$-type. Applying \cite[Corollary 2.6]{Yan96} again, we have the following result. \begin{lemma}\label{har-forms} Let $(M,\mathcal{F},\omega)$ be a transversely symplectic foliation of co-dimension $2n$. Then for any $0\leq k\leq n$,
the Lefschetz map
\[ L^k: \Omega^{n-k}_{hr}(M)\rightarrow \Omega^{n+k}_{hr}(M)\] is an isomorphism.
\end{lemma}

\begin{definition} \label{weak-lefschetz-property}
Let $(M,\mathcal{F},\omega)$ be a transversely symplectic foliation of co-dimension $2n$. It is said to satisfy the {\bf transverse Hard Lefschetz property},  if for any $0\leq k\leq n$, the Lefschetz map
  \begin{equation}\label{foliation-Lefschetz-map} L^{n-k}: H^{k}_B(M,\R)\rightarrow H_B^{2n-k}(M,\R)\, \,\,[\alpha]_B\mapsto [\omega^{n-k}\wedge \alpha]_B
  \end{equation} is an isomorphism. More generally, for any integer $0\leq s\leq n-1$, $(M,\omega, \mathcal{F})$ is said to satisfy the {\bf transverse $s$-Lefschetz property}, if for any $0\leq k\leq s$, the Lefschetz map (\ref{foliation-Lefschetz-map}) is an isomorphism.

 \end{definition}

\begin{remark}
 Let $(X,\omega)$ be a $2n$ dimensional symplectic manifold. In the literature of symplectic Hodge theory, c.f. \cite{Ma95}, \cite{Yan96}, $X$ is said to satisfy the Hard Lefschetz property if for any $0\leq k\leq n$, the Lefschetz map (\ref{Lefschetz-map}) is surjective.  By the Mathieu's theorem \cite{Ma95}, this is equivalent to saying that the Brylinski conjecture holds for $X$.  However, sometimes  the Hard Lefschetz property is also used in a slightly stronger sense, c.f. \cite{Gui01}, in which $(X,\omega)$ is said to satisfy this property if for each $0\leq k\leq n$, the Lefschetz map (\ref{Lefschetz-map}) is an isomorphism. This strengthening is necessary in order to establish the symplectic $d\delta$-lemma. Clearly, when the symplectic manifold $X$ is compact, these two versions of the Hard Lefschetz property are equivalent to each other due to the Poincar\'e duality.

Analogously, on a transversely symplectic foliation, there are two slightly different versions of the transverse Hard Lefschetz property. However, the Poincar\'e duality may not hold for basic cohomology groups even if $M$ is compact. For this reason, in Definition \ref{weak-lefschetz-property}, we use the Hard Lefschetz property in the strong sense,  so that we can establish the symplectic $d\delta$-lemma.

% Let $(M,\mathcal{F},\omega)$ be a transverse symplectic foliation of co-dimension $2n$. Then it satisfies the transverse Hard Lefschetz property if and only if it satisfies the $(n-1)$-Lefschetz property.

\end{remark}

%\begin{theorem}\label{symplectic-ddelta}(\cite{Mer98}, \cite{Gui01}, \cite{He10}) Assume that $M$ is a compact odd dimensional symplectic manifold which satisfies the transverse Hard Lefschetz property, and which admits a connection
%one form. Then on the space of basic forms, we have the following result.  \begin{equation}\label{ddelta-lemma}\text{im} d\cap \text{ker}\delta =\text{ker} d \cap \text{im} \delta=\text{im} d\delta.\end{equation}

%\end{theorem}

%Next, we present the primitive decomposition of the basic cohomology. We first define the basic version of the primitive cohomology as follows.

\begin{definition}\label{primitive1}Let $(M,\omega, \Omega)$ be a transversely symplectic foliation of co-dimension $2n$. For any $0\leq r \leq n$, the $r$-th primitive basic cohomology group, $PH_B^{r}(M,\R)$, is defined as follows.
 \[   PH_B^r (M,\R)= \text{ker}(L^{n-r+1}: H_B^r(M,\R)\rightarrow H_B^{2n-r+2}(M,\R)) .\]
\end{definition}

Finally, we collect here a few commutator relations which we will use later in this paper. We refer to \cite{He10} for a detailed proof.

\begin{lemma} \label{commutator} \[ [d, \Lambda]=\delta,\,\,\, [\delta, L]=d, \,\,\,[d\delta, L]=0,\,\,\,[d\delta,\Lambda]=0.\]

\end{lemma}

\section{ Symplectic $d\delta$-lemma on transversely symplectic foliations }
\label{ddelta-lemma}
In this section, we extend the symplectic Hodge theoretic results to transversely symplectic foliations. Analogous to the treatment used in \cite{He10}, our proof follows closely the original arguments used in \cite{Gui01}.  Nevertheless, for completeness, we have made our proof as self-contained as possible. It is worth mentioning that He proved the symplectic $d\delta$-lemma on odd dimensional symplectic manifolds under one extra condition,  that there exists a so called connection one form on the foliated manifold, c.f. \cite[Theorem 3.4.1]{He10}. In the present paper, we use a slightly different argument,  which allows us to establish the $d\delta$-lemma without this assumption.

%\begin{remark} Let $M$ be a compact odd dimensional symplectic manifold. It was shown in \cite{He10} that $[\omega^n]$ defines a non-trivial %cohomology class in

%\end{remark}

 \begin{theorem}\label{Mathieu's-theoremv2}
Let $(M,\mathcal{F},\omega)$ be a transversely symplectic foliation of co-dimnesion $2n$, and let $0\leq s \leq n - 1$. Then for any $0\leq k\leq s$, the Lefschetz map (\ref{foliation-Lefschetz-map}) is surjective
if and only if either one of the following equivalent statements holds.
\begin{itemize}

\item[1)] $H^k_{hr}(M,\omega) = H^k_B(M,\R)$, for any $0\leq k \leq s + 2$, and $
H^{2n-k}_{hr}(M,\omega) = H_B^{2n-k}(M,\R)$, for any $0\leq k \leq s$.
\item[2)] $H^{2n-k}_{hr}(M,\omega) = H_B^{2n-k}(M,\R)$ for any $0\leq k \leq s$.
\end{itemize}
\end{theorem}

\begin{proof}

\textbf{Step 1}. Clearly,  1) implies  2). Assume that 2) holds. We show that for any $0\leq k\leq s$, the Lefschetz map (\ref{foliation-Lefschetz-map}) is surjective. To this end, we consider the following commutative diagram.
\begin{equation}\label{commutative-diagram}\xymatrix{\ar @{}  [dr] \Omega_{hr}^k(M)\ar[d]\ar[r]^-{\wedge\omega^{n-k}} &\Omega_{hr}^{2n-k}(M)\ar[d]\\
H_B^k(M,\R)\ar[r]^-{\wedge[\omega^{n-k}]} &H^{2n-k}_B(M,\R)}\end{equation}
Here two vertical maps are given by the composition of the inclusion map $\Omega_{hr}^*(M)\hookrightarrow \text{ker}\left(\Omega^*_{B}(M)\xrightarrow{d}\Omega_B^{*+1}(M)\right)$ and the natural quotient map
\[\text{ker}\left(\Omega^*_{B}(M)\xrightarrow{d}\Omega_B^{*+1}(M)\right)\rightarrow H_B^*(M,\R).\]Since by Lemma \ref{har-forms}, the map $\Omega_{hr}^k(M)\xrightarrow{\wedge \omega^{n-k}} \Omega_{hr}^{2n-k}(M)$ is an isomorphism, it follows easily from Condition 2) that the bottom horizontal map must be surjective.

\textbf{Step 2}. Assume that for any $0\leq k\leq s$, the Lefschetz map (\ref{foliation-Lefschetz-map}) is surjective. We first show that  \begin{equation}\label{induction-harmonic}H^k_{hr}(M,\R) = H^k_B(M,\R),  \forall\, 0\leq k\leq s+2.\end{equation}  We claim that
for any $0\leq k\leq s+2$,  
\begin{equation}\label{primitive-decom-step1} H_B^k(M,\R)= PH_B^k(M,\R) +\text{im}\,L .\end{equation}

To see this, note that  $\forall\, [\alpha]_B\in H^k_B(M,\R)$, by the transverse $s$-Lefschetz property, there exists $[\beta]_B\in H^{k-2}_B(M,\R)$, such that $L^{n-k+1}[\alpha]_B=L^{n-k+2}[\beta]_B$. Thus $[\alpha]_B-L[\beta]_B\in PH^k_B(M,\R)$. This proves that
\[ [\alpha]_B=\left([\alpha]_B-L[\beta]_B\right)+L[\beta]_B\in PH_B^k(M,\R) +\text{im}\,L,\]
and establishes our claim.

To complete the proof, we proceed by induction on the degree of the cohomology group $H_B^k(M,\R)$. When $k=0, 1$, it is easy to see that any $0$-cycle and $1$-cycle are harmonic, which implies that  (\ref{induction-harmonic}) holds for $k=0,1$.  Assume (\ref{induction-harmonic}) is true for any $k<p\leq n$. To show that it holds for $k=p$, it suffices to show that any cohomology class 
in $PH_B^p(M,\R)$ admits a harmonic representative.

Let $[\alpha]_B\in PH^p_B(M,\R)$, where $p\leq s+2$. Then $L^{n-p+1}[\alpha]_B=0$. Thus there exists $\gamma\in \Omega^{2n-p+1}_{bas}(M)$, such that $L^{n-p+1} (\alpha)=d\gamma$. Since $L^{n-p+1}: \Omega_{bas}^{p-1}(M)\rightarrow  \Omega_{bas}^{2n-p+1}(M)$ is an isomorphism, there exists $\beta\in \Omega^{p-1}_{bas}(M)$ such that $\gamma=L^{n-p+1}(\beta)$. This implies that
$L^{n-p+1}(\alpha-d \beta)=0$. Equivalently, we have that $\Lambda(\alpha-d\eta)=0$. Thus $\delta(\alpha-d \beta )=
[d,\Lambda](\alpha-d\eta)=0$, which shows that  $[\alpha]_B$ has a harmonic representative $\alpha-d\beta$.
The other half of the statement in Condition 1) follows easily from the $s$-Lefschetz property and the commutative diagram (\ref{commutative-diagram}).
\end{proof}
\begin{remark} \label{harmonic-primitive-class} In the second step of the proof of Theorem \ref{Mathieu's-theoremv2}, we actually proved that any primitive cohomology class has a symplectic harmonic representative which is a closed primitive basic form.

\end{remark}
\begin{theorem} (c.f. \cite{Yan96})\label{primitive-decomposition} Let $(M,\mathcal{F},\omega)$ be a co-dimension $2n$ transversely symplectic foliation that satisfies the transverse $s$-Lefschetz property. Then for any $0\leq i\leq s+2$ or $2n-s\leq i\leq 2n$, we have that
\begin{equation}  \label{lef-decomp-identity} H_B^i(M,\R)=\bigoplus_r L^r PH_B^{i-2r}(M,\R)
\end{equation}\end{theorem}

\begin{proof} We first claim that  (\ref{lef-decomp-identity}) holds for $0\leq i\leq s+2$. In view of (\ref{primitive-decom-step1}),
it suffices to show that $PH_B^i(M,\R)\cap L(H^{i-2}_B(M,\R))=\{0\}$.
To see this, assume that $[\alpha]_B=L([\beta]_B)\in PH_B^i(M,\R)\cap L(H^{i-2}_B(M,\R))$. Then by primitivity, $L^{n-i+1}([\alpha]_B)=L^{n-i+2}([\beta]_B)=0$.
However, since by assumption the Lefschetz map $L^{n-i+2}: H^{i-2}_B(M,\R)\rightarrow H^{2n-i+2}(M,\R)$ is an isomorphism, it follows that $[\beta]_B=0$. This implies that $[\alpha]_B=L([\beta]_B)=0$, from which our claim follows. Since for each $0\leq k\leq s$, the Lefschetz map (\ref{foliation-Lefschetz-map}) is an isomorphism, (\ref{lef-decomp-identity}) also holds for any $2n-s\leq i\leq 2n$.

\end{proof}

\begin{definition} Let $(M,\mathcal{F},\omega)$ be a transversely symplectic foliation of co-dimension $2n$, and let $0\leq s \leq n - 1$. We say that
$(M,\mathcal{F},\omega)$ satisfies the symplectic $d\delta$-lemma up to degree $s$ if
\begin{equation}\label{s-ddelta-lemma}\begin{split} & \text{Im}\,d \cap\text{ker}\delta=\text{Im}\, d\delta=\text{Im}\, \delta\cap\text{ker}\,d ,\,\text{on }\,\Omega^{k}_{bas}(M)\,\,\,\forall\,k\leq s,\\&\text{Im}d \cap\text{ker}\delta=\text{Im}\, d\delta, \,\,\,\text{on}\, \Omega^{s+1}_{bas}(M).\end{split}
\end{equation}

\end{definition}

\begin{lemma} \label{decomp-1}Let $(M,\mathcal{F},\omega)$ be a transversely symplectic foliation of co-dimension $2n$ that satisfies the transverse $s$-Lefschetz property, and let $0\leq s \leq n - 1$. Suppose that  $\alpha\in \Omega_{bas}^k(M)$, where  $0\leq k\leq s+2$ or $2n-s\leq k\leq 2n$. Consider the Lefschetz decomposition of $\alpha$ as in Equation (\ref{lefschetz-decompose-forms}). If $\alpha$ is both harmonic and $d$-exact, then each $\alpha_r$ is $d$-exact.
\end{lemma}

\begin{proof} \textbf{Step 1}. We first prove the case when $0\leq k\leq s+2$. Write \begin{equation}\label{L-decomposition-2} \alpha=\displaystyle\sum_{r\geq  0} L^r\alpha_{r},\end{equation}
where $\alpha_{r}$ is a closed primitive basic form of degree $k-2r$. We will use downward induction on $r$. Note that when $r>\frac{k}{2}$, $\alpha_{r}=0$ is $d$-exact.
Let us assume by induction that $\alpha_r$ is $d$-exact for $r>q$, where $q>0$,  and conclude that $\alpha_q$ is $d$-exact. By the induction hypothesis,
\[ \alpha'=\alpha-\displaystyle\sum_{r> q} L^r\alpha_r=\displaystyle\sum_{r\leq q}L^r\alpha_r \] is $d$-exact.
Applying $L^{n-k+q}$ we get from the above identity that
\[ L^{n-k+q}\alpha'=L^{n-k+2q}\alpha_q+ \displaystyle\sum_{r< q}L^{n-(k-2r)+(q-r)}\alpha_r.\]
Since $\alpha_r$ is a primitive form with degree $k-2r$, we have that \[L^{n-(k-2r)+(q-r)}\alpha_r=0,\,\,\,\,\forall\,r<q.\] It follows
that
\begin{equation}\label{HLP-eq1} L^{n-k+q}\alpha'=L^{n-k+2q}\alpha_q.\end{equation}
The left hand side of (\ref{HLP-eq1}) is $d$-exact. Since $k-2q\leq s$, it follows from the transverse $s$-Lefschetz property that $\alpha_q$ is exact as well. This proves that for each $r>0$, $\alpha_r$ is $d$-exact.
Finally, if $\alpha_0$ appears in the righthand side of Equation (\ref{L-decomposition-2}), then by the above argument,
\[ \alpha_0=\alpha-\displaystyle \sum_{r>0}L^{r}\alpha_r\]
must be exact as well. This finishes the proof of the case when $0\leq k\leq s+2$.

%\begin{lemma}\label{decomp-2} Let $(M,\omega,\Omega)$ be a symplectic manifold of dimension $2n+1$, and let $s \leq n - 1$. Assume that $M$ satisfies the $s$-Lefschetz property, and let $\alpha\in \Omega_{bas}^{2n-k}(M)$ for $0\leq k\leq s$. Consider the Lefschetz decomposition of $\alpha$ as in Equation (\ref{lefschetz-decompose-forms}). If $\alpha$ is Harmonic and $d$-exact, then each $\alpha_r$ is $d$-exact.
%\end{lemma}

\textbf{Step 2.} We prove the case when $2n-s\leq k\leq 2n$.  Set $p=2n-k$. Since $\alpha \in\Omega_{bas}^{2n-p}(M)$ is harmonic, by Lemma \ref{har-forms}, there exists a harmonic $p$-form $\beta \in \Omega_{bas}^p(M)$ such that
$\alpha=L^{n-p} \beta$, where $0\leq p\leq s$. Lefschetz decompose $\beta$ as follows.
\[ \beta =\displaystyle \sum_{r\geq \text{max}(\frac{p-n}{2}, 0)}L^r\beta_{r},\]
where $\beta_{r}\in \Omega_{bas}^{p-2r}(M)$. Since $M$ satisfies the transverse $s$-Lefschetz property, $\beta$ must be a $d$-exact basic form in $\Omega_{bas}(M)$. By our work in Step 1, each $\beta_r$ must be $d$-exact.  However, we have
\[\alpha=\displaystyle \sum_{r\geq \text{max}(\frac{k-n}{2}, 0)}L^{n-p+r}\beta_{r}.\]
Since the Lefschetz decomposition of $\alpha$ is unique, this completes the proof of Lemma \ref{decomp-1}.

\end{proof}

\begin{proposition} \label{ddelta-relation}Suppose that $(M,\mathcal{F},\omega)$ is a transversely symplectic
foliation of co-dimension $2n$  that satisfies the transverse $s$-Lefschetz property, where $0\leq s \leq n - 1$.
Then we have that
\begin{itemize}

\item [a)] If $\alpha\in \Omega_{bas}(M)$ is both $d$-exact and $\delta$-closed, and if either $0\leq \text{deg}\, \alpha\leq s+2$ or $2n-s\leq \text{deg}\, \alpha\leq 2n$, then $\alpha$ is $\delta$-exact.
\item[b)] If $\alpha\in \Omega_{bas}(M)$ is both $\delta$-exact and $d$-closed, and if either $0\leq \text{deg}\, \alpha\leq s$ or $2n-s-2\leq \text{deg} \, \alpha\leq 2n$, then $\alpha$ is $d$-exact.
\item[c)]\[\,\,\,\text{on}\,\left( \bigoplus_{k=0}^s\Omega^k(M)\right)\oplus\left(\bigoplus_{k=2n-s}^{2n}\Omega^k(M)\right), \,\,\,\text{Im}\, d\cap \text{ker}\, \delta=\text{ker}\, d\cap \text{Im}\, \delta. \]

\end{itemize}

\end{proposition}

\begin{proof} We first prove a). Let $\alpha\in \Omega_{bas}^k(M)$ be both $d$-exact and $\delta$-closed, where $0\leq k\leq s+2$ or
$2n-s\leq k\leq 2n$. Lefschetz decompose $\alpha$ into
\[\alpha=\displaystyle\sum_{r\geq  0} L^r\alpha_{r},\]
where $\alpha_{r}$ is a closed primitive basic form of degree $k-2r$. Without loss of generality, we may assume that
$r\leq n-(k-2r)$, for otherwise $L^r\alpha_r=0$ due to the primitivity of $\alpha_r$.
Since $\alpha$ is both $d$-exact and $\delta$-closed, by Lemma \ref{decomp-1}, each term $\alpha_r$ is $d$-exact.
Now by Lemma \ref{Weil-identity},
\[\star \alpha=\displaystyle\sum_{r\geq  0}C_r L^{n-k+r}\alpha_{r},\] where $C_r$ is a constant that depends on $k$ and $r$.
It follows $\star \alpha$ is $d$-exact. Therefore $\alpha$ must be $\delta$-exact.

Next we prove b).  Suppose that $\alpha\in \Omega_{bas}(M)$ is both $\delta$-exact and $d$-closed, and that either $0\leq \text{deg}\, \alpha\leq s$ or $2n-s-2\leq \text{deg} \, \alpha\leq 2n$. Then $\star \alpha$ is both $d$-exact and $\delta$-closed. Moreover, either $0\leq \text{deg}\, \star\alpha\leq s+2$ or $2n-s\leq \text{deg}\, \star \alpha\leq 2n$. It follows from a) that $\star\alpha$ is $\delta$-exact.
Therefore $\alpha$ itself is $d$-exact. Finally, we note that c) is an immediate consequence of a) and b).

\end{proof}

Suppose that $\tau\in \Omega_{bas}^k(M)$ has the property that $d\tau$ is harmonic, where $0\leq k\leq s+2$ or $2n-s\leq k\leq 2n$. Lefschetz decompose $\tau$ as follows.
\begin{equation}\label{summand-1} \tau=\sum_{r\geq (p-n)_+}L^r\alpha_r,\end{equation}
where $\alpha_r$'s are primitive forms. Then
\begin{equation}\label{summand-2}d\tau=\sum_{r\geq (p-n)_+}L^rd\alpha_r.\end{equation}
 The argument given in the proof of \cite[Lemma 3.2.4, 3.2.5, Cor. 3.2.6]{He10} extends verbatim here to give us the following strengthening of Proposition \ref{ddelta-relation}.
%\begin{lemma}\label{decomp-2}
%$d\alpha_r=\beta_r+\omega\wedge \beta_r',\,\,\,\,\forall \, r,$ where $\beta_r$ and $\beta_r'$ are primitive forms.
%\end{lemma}

\begin{lemma}\label{exact1} Suppose that $(M,\mathcal{F},\omega)$ is a transversely symplectic foliation of co-dimension $2n$ that satisfies the transverse $s$-Lefschetz property, where $0\leq s \leq n - 1$. Then in Equation (\ref{summand-2}), each term $d\alpha_r=\beta_r+\omega\wedge \beta_r'$,  where $\beta_r$ and $\beta_r'$ are $d$-exact primitive basic forms.  As a result, each summand $L^r d\alpha_r$ in Equation (\ref{summand-2})
is $\delta$-exact. 
\end{lemma}

We are ready to prove the symplectic $d\delta$-lemma on basic differential forms.

\begin{theorem}\label{weak-ddelta-lemma} Let $(M,\mathcal{F},\omega)$ be a transversely symplectic foliation foliation of co-dimension $2n$ on a connected manifold $M$, and let $0\leq s \leq n - 1$.  Then the following statements are equivalent:
\begin{itemize}\item [1)] $M$ has the transverse $s$-Lefschetz property.
\item[2)] $M$ satisfies the $d\delta$-lemma up to degree $s$.
\item[3)]  The identities (\ref{s-ddelta-lemma}) hold on $\Omega_{bas}^{\geq 2n-s}(M)$, and $\text{Im} \delta\cap \text{ker} d = \text{Im} d\delta$ holds on
$\Omega_{bas}^{2n-s-1}(M)$.\end{itemize}

\end{theorem}

\begin{proof}
 \textbf{Step 1}. We show that $1)$ implies $2)$.  Since by Proposition \ref{ddelta-relation},
\[\text{Im} d\cap \text{ker} \delta=\text{ker}\, \delta\cap\text{im}d\,\,\,\text{ on}\,\,\bigoplus_{k=0}^s\Omega_{bas}^k(M),\] it suffices to show that  for any $0\leq k \leq s+1$, if a basic $k$-form $\alpha$ is both $d$-exact and $\delta$-closed, then $\alpha \in \text{im}\,d\delta$. 

We will proceed by using induction on the degree $k$ of $\alpha$. The case of degree 0 is trivial.
Assume that $\alpha\in\Omega^{1}_{\textmd{bas}}(M)$ satisfies $\alpha=df$ for some basic function $f$ and $\delta\alpha=0$.
Note that $f$ represents a homology class in $H_{\delta}(\Omega_{bas}^0(M),\R)$. However, it follows from  Lemma \ref{homology-cohomology} and the transverse $s$-Lefschetz property, that 
\[H_{\delta}(\Omega^0_{bas}(M),\R)\cong H^{2n}_B(M,\R)\cong H^0_B(M,\R)\cong \R.\] Here the last equality holds since $M$ is  connected.  Therefore there exists a constant $a$ and a basic one form $\gamma$ such that $f=a+\delta \gamma$.
It follows that $\alpha=d(a+\delta \gamma)=d\delta \gamma $.

Assume by induction that this is true for basic forms up to degree $<k$, where $k\leq s+1$. Now suppose that a $k$-form $\alpha \in \text{im}\,d \cap\text{ker}\delta$. By assumption, $\alpha= d\beta$ for some $\beta\in \Omega^{k-1}(M)$. Lefschetz decompose $\beta$ as
\[ \beta=\displaystyle \sum_{r\geq 0} L^r\beta_r,\] where $\beta_r\in \Omega_{bas}^{k-1-2r}(M)$ are primitive basic forms. Then
\[ d\beta=\displaystyle \sum_{r\geq 0} L^r d\beta_r.\]
By Lemma \ref{exact1}, $d\beta_r=v_r+Lv_r'$, where $v_r$ and $v_r'$ are exact primitive basic forms of degree $k-2r+1$ and $k-2r-1$ respectively.  As a result, each term $L^rd\beta_r$ is both $d$-exact and $\delta$-closed. So we need only to prove that
$d\beta \in \text{im}\, d\delta$ in the case that $\beta=L^r v$, where $v\in \Omega_{bas}^{k-2r-1}(M)$ is a primitive form, and $d\beta$ is $\delta$-closed.

By Lemma \ref{Weil-identity}, we get
\[ \star L^r v=CL^{n-k+r-1} v,\] where $C$ is a constant.

By Lemma \ref{exact1}, \[ d\star L^r v=C \left(L^{n-k+r+1} u_1+L^{n-k+r+2} u_2\right)\]
where $u_1$ and $u_2$ are both primitive forms, and are both $d$-exact.
Apply $\star$ on both sides of the equation again, we get
\[ \delta L^r v=C_1L^{r-1} u_1+C_2L^r u_2.\]
Hence $\delta L^rv $ is both $\delta$-exact and $d$-exact. By the induction hypothesis, there exists
$\sigma \in \Omega_{bas}^{k-2}(M)$ such that $\delta L^rv=\delta d\sigma$.

Now set $\gamma= L^rv- d\sigma$. Then $\alpha= d\gamma$ and $\delta \gamma=0$. Sine $\gamma$ is $\delta$-closed, $\star \gamma$ is a closed basic form of degree $2n-(k-1)$. Since $M$ satisfies the transverse $s$-Lefschetz property, by Theorem \ref{Mathieu's-theoremv2} there exists a harmonic form $\eta$ such that $\star \gamma-\eta$ is $d$-exact. Consequently, $\gamma-\star \eta=\star(\star \gamma-\eta)=\delta \mu$ for some $\mu\in \Omega^k_{bas}(M)$. Note that $\star \eta$ is also harmonic. It follows that
\[ \alpha =d\gamma=d(\star \eta+\delta \mu)=d\delta \mu.\]

%\linebreak
\textbf{Step 2}. We show that 2) implies 3). Note that by Proposition \ref{ddelta-relation}, we have that
\[\text{Im} d\cap \text{ker} \delta=\text{ker}\, \delta\cap\text{im}d\,\,\,\text{ on}\,\,\bigoplus_{k=2n-s}^{2n}\Omega_{bas}^k(M).\]
Thus it suffices to show that for any $0\leq k\leq s+1$, if a basic $(2n-k)$-form $\alpha$ is both $\delta$-exact and $d$-closed, then it must lie in the
image of $d\delta$. Indeed, given such a basic form $\alpha$, $\star \alpha$ must be both $d$-exact and $\delta$-closed. Since $\star\alpha$ is of degree $k$, it follows from 2)  that there exists a basic form $\eta$ such that $\star \alpha=d\delta \eta$. As a result, we have that
\[\alpha=\star(\star \alpha)= \star d\delta\eta =(\star d\star) (\star \delta \star) (\star \eta)=\pm d\delta  \eta.\]

\textbf{Step 3}. We show that 3) implies 1). We first show that for any  $0\leq k\leq s$, the Lefschetz map (\ref{foliation-Lefschetz-map}) is surjective. In view of Theorem \ref{Mathieu's-theoremv2}, it suffices to show that  any cohomology class $[\alpha]\in H^{2n-k}_B(M)$ has a harmonic representative, where $0\leq k\leq s$. Note that $\delta \alpha$ is both $\delta$-exact and $d$-closed, and has degree $2n-k-1$. It follows from 3) that $\delta \alpha=\delta d\gamma$ for some basic form $\gamma$. Note that $\delta(\alpha-d\gamma)=0$.
Thus $\alpha-d\gamma$ is both $d$-closed and $\delta$-closed, and therefore is harmonic. Clearly, $\alpha-d\gamma$ is a harmonic representative for the cohomology class $[\alpha]_B$.

Next, we show that for any  $0\leq k\leq s$, the Lefschetz map (\ref{foliation-Lefschetz-map}) is injective. Suppose that $[\alpha]_B\in H^{k}_B(M)$ such that
$L^{n-k}[\alpha]_B=0$. By the same argument as given in the previous paragraph, it is easy to see that $[\alpha]_B$ admits a harmonic representative. Without loss of generality, we may assume that $\alpha$ is a harmonic form. It follows that $L^{n-k}\alpha$ is both $d$-exact and $\delta$-closed. Thus 3) implies that $L^{n-k}\alpha = d\delta \gamma$ for some basic form $\gamma$ of degree $2n-k$. By Lemma \ref{yan's-result}, $\gamma=L^{n-k}\eta$ for some basic $k$-form $\eta$. Since $d\delta$ commutes with $L$,  \[L^{n-k}\alpha=d\delta (L^{n-k}\eta)=L^{n-k}(d\delta\eta).\] Recall that by Lemma \ref{yan's-result}, the map $L^{n-k}: \Omega_{bas}^k(M)\rightarrow \Omega^{2n-k}_{bas}(M)$ is an isomorphism. It follows that $\alpha=d\delta \eta$. This completes the proof of Theorem \ref{weak-ddelta-lemma}.
\end{proof}

\section{Review of contact and Sasakian geometry}\label{review-contact}

 Let $(M,\eta)$ be a co-oriented contact manifold with a contact one form $\eta$. We say that $(M,\eta)$ is \textbf{$K$-contact} if there is an endomorphism
 $\Phi: TM\rightarrow TM$ such that the following conditions are satisfied.
 \begin{itemize}
  \item[1)]  $\Phi^2=-Id+\xi\otimes \eta$, where $\xi$ is the Reeb vector field of $\eta$;
  \item [2)]  the contact one form $\eta$ is compatible with $\Phi$ in the sense that
  \[ d\eta(\Phi(X),\Phi(Y))=d\eta(X,Y)\]
  for all $X$ and $Y$, moreover, $d\eta(\Phi(X),X)>0$ for all non-zero $X \in \text{ker}\,\eta$;

  \item[3)] the Reeb field of $\eta$ is a Killing field with respect to the Riemannian metric defined
  by the formula
  \[ g(X,Y)=d\eta(\Phi(X),Y)+\eta(X)\eta(Y).\]

 \end{itemize}

Given a $K$-contact structure $(M,\eta,\Phi,g)$, one can define a metric cone
\[ (C(M), g_C)=(M\times \R_+, r^2g+dr^2),\] where $r$ is the radial coordinate. The $K$-contact structure $(M,\eta,\Phi)$ is called Sasakian if this metric cone is a K\"ahler manifold with K\"ahler form $\dfrac{1}{2} d(r^2\eta)$.

Let $(M,\eta)$ be a contact manifold with contact one form $\eta$ and a characteristic Reeb vector $\xi$,  let $\omega=d\eta$, and let $\mathcal{F}_{\xi}$ be the Reeb characteristic foliation.  As we explained in Example \ref{Contact-example}, $(M,\mathcal{F},\omega)$ is an one dimensional transversely symplectic foliation.  The following result   relates $H^*_B(M)$ to $H^*(M)$.

\begin{proposition}\label{exact-sequence}(\cite[Sec. 7.2]{BG08})\begin{itemize} \item[1)] On any $K$-contact manifold $(M,\eta)$, there is a long exact cohomology sequence
\begin{equation}\label{Gysin}  \cdots \rightarrow H^k_B(M,\R) \xrightarrow{i_*} H^k(M, \R)\xrightarrow{j_k} H^{k-1}_B(M,\R)\xrightarrow{\wedge[d\eta]} H^{k+1}_B(M,\R)\xrightarrow{i_*} \cdots,\end{equation} where
$i_*$ is the map induced by the inclusion, and $j_k$ is the map induced by $\iota_{\xi}$.

\item[2)] If $(M,\eta)$ is a compact $K$-contact manifold of dimension $2n+1$, then for any $r\geq 0$ the basic cohomology $H_B^r(M,\R)$ is finite dimensional, and for $r>2n$,
the basic cohomology $H_B^r(M,\R)=0$; moreover, for any $0\leq r\leq 2n$, there is a non-degenerate pairing
\[ H^r_B(M,\R)\otimes H^{2n-r}_B(M,\R)\rightarrow \R,\,\,\,([\alpha]_B,[\beta]_B)\mapsto \int_M\, \eta\wedge\alpha\wedge\beta. \]

\end{itemize}

\end{proposition}
On a compact Sasakian manifold $M$, the following Hard Lefschetz theorem is due to El Kacimi-Alaoui \cite{ka90}.

\begin{theorem}(\cite{ka90})\label{trans-Kahler} Let $(M,\eta,g)$ be a compact Sasakian manifold with a contact one form $\eta$ and a Sasakian metric $g$. Then $M$ satisfies the transverse Hard Lefschetz property introduced in Definition \ref{weak-lefschetz-property}.
\end{theorem}

More recently,  Cappelletti-Montano, De Nicola, and Yudin \cite{CNY13} established a Hard Lefschetz theorem for the De Rham cohomology group of a compact Sasakian manifold.

\begin{theorem}(\cite{CNY13}) \label{HLP-sasakian} Let $(M,\eta,g)$ be a $2n+1$ dimensional compact Sasakian manifold with a contact one form $\eta$ and a Sasakian metric $g$, and let $\Pi: \Omega^*(M)\rightarrow \Omega_{har}^*(M)$ be the projection onto the space of Harmonic forms. Then for any $0\leq k\leq n$, the map
\[Lef_k: H^{k}(M,\R)\rightarrow H^{2n+1-k}(M,\R), [\beta]\mapsto [\eta\wedge (d\eta)^{n-k}\wedge \Pi \beta]\]
is an isomorphism. Moreover, for any $[\beta]\in H^k(M,\R)$, and for any closed basic primitive $k$-form $\beta'\in [\beta]$, $[\eta\wedge (d\eta)^{n-k}\wedge \beta']=Lef_k([\beta])$. In particular, the Lefschetz map $Lef_k$ does not depend on the choice of a compatible Sasakian metric.

\end{theorem}

This result motivates them to propose the following definition of the Hard Lefschetz property for a contact manifold.

\begin{definition}\label{HLP-contact} Let $(M,\eta)$ be a $2n+1$ dimensional compact contact manifold with a contact $1$-form $\eta$. For any $0\leq k\leq n$,
define the Lefschetz relation between the cohomology group $H^{k}(M,\R)$ and $H^{2n+1-k}(M,\R)$ to be
\begin{equation}\label{Lef-relation} \mathcal{R}_{Lef_k}=\{([\beta],[\eta\wedge L^{n-k}\beta])\,\vert \iota_{\xi}\beta=0, d\beta=0, L^{n-k+1}\beta=0\}.\end{equation}
If it is the graph of an isomorphism $Lef_k: H^{k}(M,\R)\rightarrow H^{2n+1-k}(M,\R)$ for any $0\leq k\leq n$, then
the contact manifold $(M,\eta)$ is said to have the hard Lefschetz property.

\end{definition}

We introduce the following refinement of Definition \ref{HLP-contact}.

\begin{definition} \label{HLP-contactv2} Let $(M,\eta)$ be a $2n+1$ dimensional compact contact manifold with a contact $1$-form $\eta$, and let $0\leq s\leq n-1$.  If for any $0\leq k\leq s$, (\ref{Lef-relation}) is the graph of an isomorphism $Lef_k: H^{k}(M,\R)\rightarrow H^{2n+1-k}(M,\R)$, then
the contact manifold $(M,\eta)$ is said to have the $s$-Lefschetz property.

\end{definition}

\begin{remark} Note that by the Poincar\'e duality, every compact contact manifold is $0$-Lefschetz. For the same reason, a simply-connected compact contact manifold is $1$-Lefschetz.

\end{remark}

\section{K-contact manifolds with the transverse $s$-Lefschetz property}\label{Kcontact-s-lefschetz}

Throughout this section, we assume $(M,\eta)$ to be a $2n+1$ dimensional compact $K$-contact manifold with a contact $1$-form $\eta$, and a Reeb vector field $\xi$. Let $\omega=d\eta$, and let $\mathcal{F}_{\xi}$ be the Reeb characteristic foliation. We will apply the machinery developed in Section \ref{ddelta-lemma} to the transverse symplectic flow $(M,\mathcal{F}_{\xi},\omega)$, and prove that $M$ satisfies the transverse $s$-Leschetz  property if and only if it satisfies the $s$-Lefschetz property stated in Definition \ref{HLP-contactv2}.

\begin{lemma}\label{tech-lemma1}Let $(M,\eta)$ be a $2n+1$ dimensional compact $K$-contact manifold with a contact $1$-form $\eta$.  Assume that $M$ satisfies the transverse $s$-Lefschetz property, $0\leq s\leq n-1$. Then for any $0\leq k\leq s+1$, the  map \[ i_*: H^k_B(M, \R)\rightarrow H^k(M,\R)\] is surjective; moreover, its image equals
\begin{equation} \label{image} \{i_*[\alpha]_B\,\vert\, \alpha \in \Omega^k_{bas}(M), d\alpha=0, \omega^{n-k+1}\wedge \alpha=0\}.\end{equation}
As a result, the restriction map $i_*: PH^k_{B}(M,\R)\rightarrow H^k(M,\R)$ is an isomorphism.
\end{lemma}

\begin{proof} Consider the long exact sequence (\ref{Gysin}). By assumption, $M$ satisfies the transverse $s$-Lefschetz property. Thus  the map \[H^{i}_B(M,\R)\xrightarrow{\wedge[\omega]} H_B^{i+2}(M,\R)\] is injective for any
$ 0\leq i \leq s$. It then follows from the exactness of the sequence (\ref{Gysin}) that the map

\[ i_*: H^{k}_B(M,\R)\rightarrow H^{k}(M,\R)\] is surjective for any  $ 0\leq k\leq s+1$. This proves the first assertion in Lemma \ref{tech-lemma1}.

Since $M$ satisfies the transverse $s$-Lefschetz property, by Theorem \ref{primitive-decomposition}, for any $0\leq k\leq s+1$,
\[ H^{k}_B(M,\R)= PH^k_B(M)\oplus L H^{k-2}_B(M,\R).\]

It is clear from the exactness of the sequence (\ref{Gysin}) that
\[i_*\left(H^k_B(M,\R)\right)=i_*\left(PH^k_B(M,\R)\right),\,\,\,\text{ker} i_* \cap PH^k(M,\R)=0.\]
Therefore the restriction map $i_*:PH^k_B(M,\R)\rightarrow H^k(M,\R)$ is an isomorphism. Finally, the fact that $i_*\left( H^k_B(M,\R)\right)$ equals (\ref{image}) follows easily from Remark \ref{harmonic-primitive-class}. This completes the proof of Lemma \ref{tech-lemma1}.

\end{proof}

%\begin{remark} The result proved in Lemma \ref{tech-lemma1} is known to hold for compact Sasakian manifolds, c.f. \cite[Prop. 7.4.13]{BG08}.
%The traditional proof uses Riemannian Hodge theory associated to a compatible Sasakian metric.
%\end{remark}

 We are ready to define the Lefschetz map on the cohomology groups. In \cite{CNY13}, such maps are introduced using Riemannian Hodge theory associated to a compatible Sasakian metric. In contrast, we define these maps here using the symplectic Hodge theory on the space of basic forms.

  For any $ 0\leq k\leq s+1$, define $Lef_k : H^{k}(M,\R)\rightarrow H^{2n+1-k}(M,\R)$ as follows.  For any cohomology class
$[\gamma] \in H^{k}(M,\R)$, by Lemma \ref{tech-lemma1} there exists a closed primitive basic $k$-form $\alpha \in \mathcal{P}_{bas}^k(M)$ such that $i_*[\alpha]_B=[\gamma]$. Observe that $d \left( \eta\wedge L^{n-k}\wedge \alpha\right)= L^{n-k+1} \alpha=0$. We define \begin{equation}\label{main-map} Lef_k[\gamma]= [\eta\wedge L^{n-k} \alpha].\end{equation}

\begin{lemma}\label{tech-lemma2} Assume that $M$ satisfies the transverse $s$-Lefschetz property. Then for any $0\leq k\leq s+1$, the map (\ref{main-map}) does not depend on the choice of closed primitive basic forms.
\end{lemma}

\begin{proof} Suppose that there are two closed primitive basic $k$-forms $\alpha_1$ and $\alpha_2$ such that
$i_*[ \alpha_1 ]_B=i_*[\alpha_2]_B\in H^{k}(M,\R)$. It follows from the exactness of the sequence (\ref{Gysin}) that
$ [\alpha_1]_B=[\alpha_2]_B +L [\beta]_B$ for some
closed basic $(k-2)$-form $\beta$. Since $M$ satisfies the transverse $s$-Lefschetz property, by Theorem
\ref{Mathieu's-theoremv2} we may well assume that $\beta$ is symplectic harmonic.

Therefore,  $\alpha_1-\alpha_2 -L \beta$ is both $d$-exact and $\delta$-closed. By Theorem \ref{weak-ddelta-lemma}, the symplectic $d\delta$-lemma,
there exists a basic $k$-form $\varphi$ such that \begin{equation} \label{difference} \alpha_1-\alpha_2 -L \beta= d\delta \varphi \end{equation}

Lefschetz decompose $\beta$ and $\varphi$ as follows.
\[ \begin{split}  &\beta=\beta_{k-2}+L\beta_{k-4}+L^2\beta_{k-6}+\cdots \\
& \varphi=  \varphi_{k}+L\varphi_{k-2}+L^2\varphi_{k-4}+\cdots
\end{split} \] Here $\varphi_{k-i}\in \mathcal{P}_{bas}^{k-i}(M)$, $i=0, 2,\cdots$,  and $\beta_{k-i}\in \mathcal{P}_{bas}^{k-i}(M)$, $i=2,4,\cdots$.
Since $d\delta$ commutes with $L$, it follows from (\ref{difference}) that
\[ \alpha_1-\alpha_2=d\delta \varphi_k +L(\beta_{k-2}+d\delta \varphi_{k-2})+L^2(\beta_{k-4}+d\delta\varphi_{k-4})\cdots  .\]

Since $d\delta$ commutes with $\Lambda$, $d\delta$ maps primitive forms to primitive forms. It then follows from the uniqueness of the Lefschetz decomposition that
\[\alpha_1-\alpha_2=d\delta \varphi_k.\]

Observe that \[\begin{split} \eta\wedge \left(\omega^{n-k}\wedge (\alpha_1-\alpha_2)\right)&=
\eta\wedge \left(\omega^{n-k}\wedge d\delta\varphi_k\right)\\&=
-d\left( \eta\wedge \omega^{n-k}\wedge \delta\varphi_k\right)+ \left(L^{n-k+1} \delta\varphi_k\right).
\end{split}\]

Now using the commutator relation $[L, \delta]=-d$ repeatedly, it is clear that $L^{n-k+1} \delta\varphi_k$
must be $d$-exact, since $\varphi_k$ is a primitive $k$-form and so $L^{n-k+1}\varphi_{k}=0$. It follows immediately
that $\eta\wedge L^{n-k} (\alpha_1-\alpha_2)$ must be $d$-exact. This completes the proof of Lemma \ref{tech-lemma2}.

\end{proof}

\begin{theorem} \label{main-result1}Let $M$ be a $2n+1$ dimensional compact $K$-contact manifold with a contact one form $\eta$, and let $0\leq s\leq n-1$. Then it satisfies the transverse $s$-Lefschetz property as introduced in Definition \ref{weak-lefschetz-property} if and only if it satisfies the $s$-Lefschetz property as introduced in Definition \ref{HLP-contactv2}.
\end{theorem}

\begin{proof} {\bf Step 1.} \,Assume that $M$ satisfies the transverse $s$-Lefschetz property. We show that $M$ satisfies the $s$-Lefschetz property. Since $M$ is oriented and compact, in view of the Poincar\'e duality, it suffices to show that for any $0\leq k\leq s$, the map given in (\ref{main-map}) is injective.

Suppose that $Lef_k[\gamma]=[\eta\wedge L^{n-k} \alpha]=0$, where $\alpha \in \mathcal{P}_{bas}^k(M)$ such that $d\alpha=0$,  $i_*[\alpha]_B=[\gamma]$.
Since the group homomorphism $j_{2n+1-k}:H^{2n+1-k}(M,\R)\rightarrow H_B^{2n-k}(M,\R)$ is induced by $\iota_{\xi}$, it follows that
\[0=j_{2n+1-k}(0)=j_{2n+1-k}([\eta\wedge (L^{n-k} \alpha)])=[ L^{n-k} \alpha]_B.\]

Since $M$ has the transverse $s$-Lefschetz property, and since $0\leq k\leq s$, $[\alpha]_B=0$. Thus $[\gamma]=i_*([\alpha]_B)=0$.

{\bf Step 2.}\, Assume that $M$ satisfies the $s$-Lefschetz property, i.e., , for any $0\leq k\leq s$,
\[ \mathcal{R}_{Lef_k}=\{([\beta],[\eta\wedge L^{n-k}\beta])\,\vert \iota_{\xi}\beta=0, d\beta=0, L^{n-k+1}\beta=0\}\]
is the graph of an isomorphism $Lef_k: H^k(M)\rightarrow H^{2n-k+1}(M)$. In particular, it implies that
%\begin{itemize} \item [a)] every cohomology class in $H^k(M)$ is represented by a closed primitive basic $k$-form, where $0\leq k\leq s$;
 if $\alpha$ is a closed primitive basic $k$-form, $0\leq k\leq s$, then $Lef_k([\alpha])=[\eta\wedge L^{n-k}\alpha]$.

We first claim that if $[\alpha]_B\in PH^k_B(M,\R)$ such that  $L^{n-k}[\alpha]_B=0\in H^{2n-k}_B(M,\R)$, then $[\alpha]_B\in \text{im}\,L$. By Remark \ref{harmonic-primitive-class}, we may assume that $\alpha$ is a closed primitive basic $k$-form. Then for any closed primitive basic $k$-form $\beta$,
\[j_{2n+1}\left(Lef_k(i_*[\alpha]_B\cup i_*[\beta]_B)\right)= j_{2n+1}([\eta\wedge L^{n-k}\alpha \wedge \beta])= L^{n-k}[\alpha]_B\cup [\beta]_B=0.\]

We observe that the map $j_{2n+1}: H^{2n+1}(M,\R)\rightarrow H_B^{2n}(M,\R)$ is an isomorphism.
Indeed, it is an immediate consequence of the exactness of the sequence (\ref{Gysin}) at stage $2n+1$. Since $H_B^{i}(M,\R)=0$ when $i\geq 2n+1$, we have that
\begin{equation}\label{Gysin-final-stage}\cdots \rightarrow 0 \xrightarrow{i_*} H^{2n+1}(M, \R)\xrightarrow{j_{2n+1}} H^{2n}_B(M,\R)\xrightarrow{\wedge[\omega]} 0\rightarrow  \cdots
\end{equation}

As a result,  $Lef_k(i_*[\alpha]_B) \cup i_*[\beta]_B=0$.  Since $\beta$ is arbitrarily chosen, by the Poincar\'e duality, we must have $Lef_k[i_*[\alpha]_B)=0$. Since $Lef_k$ is an isomorphism, $i_*[\alpha]_B=0$. By the exactness of the sequence (\ref{Gysin}), $[\alpha]_B=L[\lambda]_B$ for some $[\lambda]_B\in H_B^{k-2}(M,\R)$. This proves our claim.

Now we show that for any $0\leq k\leq s$, the map
\begin{equation}\label{induction}  L^{n-k}: H^k_B(M)\rightarrow H_B^{2n-k}(M),\,\,\,[\alpha]_B\mapsto [\omega^{n-k}\wedge \alpha]_B\end{equation} is an isomorphism by induction on $k$. By Part 2) in Proposition \ref{exact-sequence}, it suffices to show that for any $0\leq k\leq s$, the map
(\ref{induction}) is injective.

Note that when $k=0,1$, for degree reasons, every closed basic $k$-form is primitive; furthermore, $\text{im}\, L\cap \Omega^k_{bas}(M)=\{0\}$. As a result, when $k=0,1$, the injectivity of the map (\ref{induction}) is a simple consequence of the claim we established above.

Assume that $M$ satisfies the transverse $p$-Lefschetz property for $p<k$. Then it follows from Theorem \ref{primitive-decomposition} that $H^k_B(M,\R)=PH^k_B(M,\R)\oplus\text{im}\,L$. Suppose that $L^{n-k}([\alpha]_B+L[\sigma]_B)=0$, where $[\alpha]_B\in PH_B^{k}(M,\R)$ and $[\sigma]\in H_B^{k-2}(M,\R)$. Then we must have
$L^{n-k+1}([\alpha]_B+L[\sigma]_B)=L^{n-k+2}[\sigma]_B=0$ since $[\alpha]_B\in PH^k_B(M)$. It follows from our inductive hypothesis again that
$[\sigma]_B=0$. As a result, $L^{n-k}[\alpha]_B=0$. By the claim we established earlier, we must have that $[\alpha]_B=L[\beta]_B$ for some $[\beta]_B\in H^{k-2}_B(M,\R)$. However, by Theorem \ref{primitive-decomposition}, $PH^k_B(M,\R)\cap \text{im}L=0$. Therefore we must have $[\alpha]_B=0$. This completes the proof of Theorem \ref{main-result1}.

\end{proof}
%\begin{remark}\label{HLP-final-stage} Let  $(M,\eta, \xi)$ be a $2n+1$ dimensional compact $K$-contact manifold that satisfies the transverse Hard Lefschetz property. Then the argument used in the Step 1 of the proof of Theorem \ref{main-result1} shows that the map
%\[ H^n(M,\R)\rightarrow H^{n+1}(M,\R),\,\,\,\,i_*[\alpha]_B \mapsto [\alpha\wedge \eta] \]
%is an isomorphism, where $\alpha \in \Omega_{bas}^n(M)$ is a closed primitive basic $n$-form. Thus $(M,\eta,\xi)$ satisfies the Hard Lefschetz property as introduced in Definition \ref{HLP-contact}.

%\end{remark}

\section{Cup length of Lefschetz $K$-contact manifolds}\label{cup-length}

In this section, we show that for compact $K$-contact manifolds,  the weak Lefschetz condition implies a fairly general result on the vanishing of cup products.
%\subsection{ Parity of lower dimensional odd Betti numbers}

%\end{proof}

\begin{theorem}\label{vanishing-cup-prod}Let $(M,\eta)$ be a $2n + 1$ dimensional compact $K$-contact manifold that satisfies the $s$-Lefschetz property, where $0\leq s\leq n-1$, and let $y_i\in H^{k_i}(M,\R)$,$1\leq i\leq p$. If $\,\forall\, 1\leq i \leq p$, $1\leq k_i\leq s+1$, and if $k_1+\cdots+k_p\geq 2n-s$, then the cup product \begin{equation} \label{vanishing-cup-eq}y_1\cup\cdots\cup y_p=0.\end{equation}
 \end{theorem}
\begin{proof}  $\forall\, 1\leq i\leq p$, since $1\leq k_i\leq s+1$,  by Lemma \ref{tech-lemma1} there exists $[\alpha_i]_B \in H_B^{k_i}(M)$,  such that $i_*([\alpha_i]_B)=y_i$.  By assumption, $M$ satisfies the $s$-Lefschetz property.  It follows from Theorem \ref{main-result1}, it must satisfy the transverse $s$-Lefschetz property as well. Now that $k_1+\cdots +k_p\geq 2n-s$, there exists $ [\beta]_B \in H_B^{2n-k_1-\cdots-k_p}(M)$, such that
\[ [\alpha_1 \wedge \cdots \wedge \alpha_p]_B= L^{k_1+\cdots +k_p-n}([\beta]_B).\]
Clearly, we have that $k_1+\cdots+k_p-n\geq n-s\geq 1$. It follows immediately from the exactness of the Sequence (\ref{exact-sequence}) that
\[ \begin{split} y_1\cup \cdots \cup y_p&=i_*([\alpha_1]_B)\cup \cdots \cup i_*([\alpha_p]_B)
\\&=i_*([\alpha_1\wedge \cdots \wedge \alpha_p]_B)\\&=i_*(L^{k_1+\cdots k_p-n}([\beta]_B))=0.\end{split}\]

\end{proof}

The following result is an immediate consequence of Theorem \ref{vanishing-cup-prod}.

\begin{theorem}\label{cup-length-lefschetz} Let $M$ be  a $2n+1$ dimensional compact Lefschetz $K$-contact manifold that satisfies the $s$-Lefschetz property, where $0\leq s\leq n-1$. Then the cup length of $M$ is $\leq 2n-s$. 

\end{theorem}

\begin{proof} Let $y_i \in H^{k_i}(M)$, $1\leq i\leq p$.  To establish Theorem \ref{cup-length-lefschetz}, it suffices to show that if $p>2n-s$, and if $k_i\geq 1$, then (\ref{vanishing-cup-eq}) holds.  Indeed, if there is a cohomology class, say $y_1$, such that
$k_1>s+1$, then we have that
\[ k_1+\cdots+k_p> s+1+ (2n-s)= 2n+1.\]
So in this case (\ref{vanishing-cup-eq}) holds for degree reasons. Therefore we may assume that $1\leq k_i\leq s+1$, $\forall\, 1\leq i\leq p$. Since $k_1+\cdots+k_p\geq p> 2n-s$,  in this case (\ref{vanishing-cup-eq}) follows directly from Theorem \ref{vanishing-cup-prod}.

\end{proof}

Since  by the Poincar\'e duality, any compact connected $K$-contact manifold is $0$-Lefschetz, Theorem \ref{cup-length-lefschetz} immediately implies the following result of Boyer and Galicki \cite[Theorem 7.4.1]{BG08}.

\begin{corollary}\label{cup-length-k-contact} The cup length of a $2n+1$ dimensional compact  connected $K$-contact manifold is at most $2n$. \end{corollary}

%If a compact connected $K$-contact manifold $M$ is also simply-connected, then $H^1(M)=0$.  In this case,  we deduce from Theorem \ref{vanishing-cup-prod} the following result.  

%\begin{corollary}\label{cup-length-s-connected-K-contact}  The cup length of a $2n+1$ dimensional compact  connected and simply-connected $K$-contact manifold is at most $n$
%\end{corollary}
%\begin{proof} Suppose that there exist $y_1,\cdots y_p\in H(M)$, such that $y_i\in H^{k_i}(M)$ for some $k_i\geq 2$, and such that $p\geq n+1$. To prove Corollary \ref{cup-length-s-connected-K-contact}, it suffices to show that $y_1\cup\cdots \cup y_p=0$. If for some $1\leq i\leq p$, $k_i\geq 3$, then $k_1+\cdots +k_p\geq 3+ 2(p-1)\geq 3+2n$. So $y_1\cup \cdots \cup y_p=0$ for degree reasons. If $\forall\, 1\leq i\leq p$, $k_i=2$, then it follows from Theorem \ref{vanishing-cup-prod} that $y_1\cup\cdots \cup y_p=0$, . 
%\end{proof}

On the other hand, by Theorem \ref{trans-Kahler}, any $2n+1$ dimensional compact Sasakian manifold satisfies the transverse hard Lefschetz property. Theorem \ref{cup-length-lefschetz} gives us the following upper bound for the cup length of a compact Sasakian manifold.

\begin{corollary}\label{cup-length-sasakian} The cup length of a $2n+1$ dimensional compact Sasakian manifold is at most $n+1$.

\end{corollary}

\section{Boothby-Wang fibration over weakly Lefschetz symplectic manifolds}\label{main-examples}

In this section, we apply the main result obtained in Section \ref{Kcontact-s-lefschetz} to a Boothy-Wang fibration, and use it to construct examples of $K$-contact manifolds without any Sasakian structures in dimension $\geq 9$.

We first briefly review Boothby-Wang construction here, and refer to \cite{B76} for more details.  A co-oriented contact structure on a $2n+1$ dimensional compact manifold $P$ is said to be regular if it is given as the kernel of a contact one form $\eta$, whose Reeb field $\xi$ generates a free effective $S^1$ action on $P$. Under this assumption,  $P$ is the total space of a principal circle bundle $ \pi: P\rightarrow M:=P/S^1$, and the base manifold $M$ is equipped with an integral symplectic form $\omega$ such that $\pi^* \omega =d\eta$.  Conversely, let $(M,\omega)$ be a compact symplectic manifold with an integral symplectic form $\omega$, and let $\pi:P\rightarrow M$ be the principal circle bundle over $M$ with Euler class $[\omega]$ and a connection one form $\eta$ such that $\pi^*\omega=d\eta$. Then $\eta$ is a contact one form on $P$ whose characteristic Reeb vector field generates the right translations of the structure group $S^1$ of this bundle.  It is easy to deduce the following result as a direct consequence of Theorem \ref{main-result1}.

\begin{theorem} \label{regular-Lef-contact}Let $\pi: P\rightarrow M$ be a Boothby-Wang fibration as we described above. Then $(P,\eta)$ satisfies the $s$-Lefschetz property if and only if the base symplectic manifold $(M,\omega)$ satisfies the $s$-Lefschetz property.
\end{theorem}
Next, we recall an useful result \cite{Ha13} on when  the total space of a Boothby-Wang fibration is simply-connected. Let $X$ be a compact and oriented manifold of dimension $m$. We say that $c \in H^2(X,\Z)$ is indivisible if the map
\[ c \cup : H^{m-2}(X,\Z)\rightarrow H^m(X,\Z)\] is surjective.
\begin{lemma} \label{boothby-wang-l1} (\cite[Lemma 15]{Ha13}) Let $\pi: P\rightarrow M$ be a Boothby-Wang fibration, and let $\omega$ be an integral symplectic form on $M$ which represents the Euler class of the Boothby-Wang fibration. Then $P$ is simply-connected if and only if $M$ is simply-connected, and the Euler class $[\omega]$ is indivisible.\end{lemma}

We also need the following result concerning the existence of symplectic manifolds which are $s$-Lefschetz but not $(s+1)$-Lefschetz proved in \cite[Prop. 5.2]{FMU07}.

\begin{theorem} \label{example-weak-lef}Let $s \geq 2$ be an even integer. Then there is a simply-connected symplectic $(W_s,\omega)$ of dimension $2(s+2)$ which is $s$-Lefschetz but not $(s+1)$-Lefschetz. Moreover, the symplectic form $\omega$ is integral, and $b_{s+1}(W_s)=3$.
\end{theorem}

\begin{remark} By \cite[Theorem 4.2]{FMU07}, the symplectic form on  $M_s$ constructed in \cite[Prop. 5.1]{FMU07} can chosen to be integral. A careful reading of the proof of \cite[Prop.5.2]{FMU07} shows that the symplectic form on $W_s$ can also chosen to be integral; moreover, $b_{s+3}(W_s)=3$. Thus by the Poincar\'e dulaity, we have that $b_{s+1}(W_s)=3$.
\end{remark}

We are ready to prove the main result of Section \ref{main-examples}.

\begin{theorem}\label{high-dim-example-1} For any even integer $s\geq 2$, there exists a $2s+5$ dimensional simply-connected compact $K$-contact manifold $(M,\eta)$, such that $(M,\eta)$ is $s$-Lefschetz but not $(s+1)$-Lefschetz, and such that $b_{s+1}(M)\leq 3$ is odd. In particular, $M$ does not support any Sasakian structure.
\end{theorem}

\begin{proof} By Theorem \ref{example-weak-lef}, there is a closed simply-connected symplectic manifold $(W_s,\omega)$ of dimension $2(s+2)$ that is $s$-Lefschetz but not $(s+1)$-Lefschetz. Moreover, the symplectic form $\omega$ is integral, and $b_{s+1}(W_s)=3$. Without loss of generality, we may also assume that $[\omega]$ is indivisible. Let $(M,\eta)$ be the Boothby-Wang firbation over
$(W_s,\omega)$ whose Chern class is $[\omega]$. Then by Lemma \ref{boothby-wang-l1}, $M$ is simply-connected.  Consider the following portion of the Gysin sequence for the principal circle bundle $\pi:M\rightarrow W_s$.
\begin{equation}\label{Gysin-2}  \cdots H^{s-1}(W_s,\R)  \xrightarrow{\wedge[\omega]} H^{s+1}(W_s,\R) \xrightarrow{\pi^*} H^{s+1}(M, \R)\xrightarrow{\pi_*} H^{s}(W_s,\R)\xrightarrow{\wedge[\omega]} H^{s+2}(W_s,\R)\xrightarrow{\pi^*} \cdots,\end{equation} where
$\pi_*:H^*(M,\R)\rightarrow H^{*-1}(X,\R)$ is the map induced by integration along the fibre.

Since $W_s$ is $s$-Lefschetz with $s$ being an even integer,  by \cite[Prop. 2.6]{FMU07}, $b_{s-1}(W_s)$ must be even as $s-1$ is odd. Moreover, the map  $H^{j}(W_s,\R)\xrightarrow{\wedge[\omega]} H^{j+2}(W_s,\R)$ must be injective for $j=s-1$ and $j=s$. As a result, $b_{s+1}(M)=b_{s+1}(W_s)-b_{s-1}(W_s)=3-b_{s-1}(W_s)$ must be odd as well. It follows from \cite[Theorem 7.4.11]{BG08} that $M$ can not support any Sasakian structure.
\end{proof}

%\begin{theorem}\label{high-dim-example-2} For any $n\geq 4$, there exists a  simply-connected $K$-contact manifold $(M,\eta)$ of dimension $2n+1$ such that $b_3(M)=3$. In particular, $M$ does not support any Sasakian structure.
%\end{theorem}

%\begin{proof} By Corollary \ref{example-weak-lef2}, for any $n\geq 4$, there exists simply-connected symplectic $(N,\omega)$ of dimension $2n$ which is $2$-Lefschetz but not $3$-Lefschetz. Moreover, the symplectic form $\omega$ is integral, and $b_{3}(N)=3$. We may well assum that $[\omega]$ is indivisible. Let $(M,\eta)$ be the Boothby-Wang firbation over $(W_s,\omega)$ whose Chern class is $[\omega]$. By Lemma \ref{boothby-wang-l1}, $M$ is simply-connected. The same argument as given in the proof of Theorem \ref{high-dim-example-1} shows that $b_3(M)=b_3(N)-b_1(N)=b_2(N)=3$.  This completes the proof.
%\end{proof}

\medskip

\noindent
Yi Lin \\
Department of Mathematical Sciences \\
Georgia Southern University\\
203 Georgia Ave., Statesboro, GA, 30460 \\
{\em E\--mail}: yilin@georgiasouthern.edu

\noindent
\noindent

\end{document}